\documentclass[suppldata]{interact}

\usepackage{epstopdf}
\usepackage{amsmath}
\usepackage[caption=false]{subfig}
\usepackage{amsfonts}

\usepackage{hyperref}
\usepackage{bbm}

\theoremstyle{plain}
\newtheorem{theorem}{Theorem}[section]
\newtheorem{lemma}[theorem]{Lemma}

\newtheorem{mydef*}{Definition*}

\theoremstyle{definition}

\newtheorem{example}[theorem]{Example}

\theoremstyle{remark}
\newtheorem{remark}{Remark}

\begin{document}
\title{C(X) determines X - an inherent theory}
\author{
\name{Biswajit Mitra and Sanjib Das$^1$}
\thanks{CONTACT Email:  bmitra@math.buruniv.ac.in \&ruin.sanjibdas893@gmail.com}
\affil{Department of Mathematics, The University of Burdwan, India 713104
\& Department of Mathematics, The University of Burdwan, India 713104}}
\thanks{$^1$The second author's research grant is supported by CSIR, Human Resource
Development group, New Delhi-110012, India }

\maketitle

\begin{abstract}
One of the fundamental problem in rings of continuous function is to extract those spaces for which $C(X)$ determines $X$, that is to investigate $X$ and $Y$ such that $C(X)$ isomorphic with $C(Y)$ implies $X$ homeomorphic with $Y$. The development started back from Tychonoff who first pointed out inevitability of Tychonoff space in this category of problem. Later S.Banach and M. Stone proved independently with slight variance, that if $X$ is compact Hausdorff space, $C(X)$ also determine $X$. Their works were maximally extended by E. Hewitt by introducing realcompact spaces and later Melvin Henriksen and Biswajit Mitra solved the problem for locally compact and nearly realcompact spaces. In this paper we tried to develop an inherent theory of this problem to cover up all the works in the literature introducing a notion so called $\mathcal{P}$-compact spaces.
\end{abstract}


\section{introduction}

In this paper we tried to develop an inherent theory of $C(X)$ determining $X$ in the sense that under what general condition(s), $C(X)$ isomorphic with $C(Y)$ implies that $X$ is homeomorphic with $Y$. It is very age-old and fundamental problem in the area of rings of continuous functions. Tychonoff proved that for any space $X$, there exists a Tychonoff space $Y$ such that $C(X)$ is isomorphic with $C(Y)$. Thus Tychonoff directed us to investigate this problem within the class of Tychonoff spaces, we shall therefore always stick ourselves within the class of Tychonoff spaces unless otherwise mentioned. Banach and Stone independently proved with slight variance that if $X$ is compact and Hausdorff space, $C(X)$ determines $X$. In the year 1948, E hewitt inroduced realcompact space and proved that within the class of realcompact space, $C(X)$ determines $X$ [\cite{gj60} Theorem 8.3] . His theory in some sense maximal as because he had shown that for any space $X$, $C(X)$ is isomorphic with $C(\upsilon X)$ [\cite{gj60}, Theorem 8.8(a)] , where $\upsilon(X)$ is the Hewitt-realcompactification of $X$. This results particularly highlights that we can never extend ``C(X) determining $X$"-type statement to the generalized notion of realcompact spaces, like nearly realcompact etc. That is why, in the year 2005, Henriksen and Mitra proved that if $X$ is locally compact and nearly realcompact, then $C(X)$ determines $X$.

In this paper we tried to build a common theory which includes all the developments so far. We have defined a subset $\mathcal{P}$ of the family of maximal ideals to be algebraic. Members of $\mathcal{P}$ are called $\mathcal{P}$-maximal ideal. A space $X$ is $\mathcal{P}$-compact if every $\mathcal{P}$-maximal ideal is fixed. A space $X$ is called locally $\mathcal{P}$ if every fixed maximal ideal is $\mathcal{P}$-maximal ideal. Both the notion of $\mathcal{P}$-compact and locally $\mathcal{P}$ for an algebraic set $\mathcal{P}$ are topological. We finally proved that if $C(X)$ is isomporphic with $C(Y)$ and both $X$ and $Y$ are $\mathcal{P}$-compact and locally $\mathcal{P}$, then $X$ is homeomorphic with $Y$.

\section{Preliminaries}
We used most of the preliminary ideas, symbols
and terminologies from the classic monograph of Leonard Gillman and Meyer Jerison,
Rings of Continuous Functions \cite{gj60}. Although we recall some   notations and concepts which are  used several times over here. For any $f\in C(X)$
or $C^{*}(X)$, $Z(f)=\{x\in X:f(x)=0\}$, called zero set of $f$ and
the complement of zero set is called cozero set or cozero part of $f$, denoted as $cozf$. For $f\in C(X)$, the set $cl_{X}(X\backslash Z(f))$ is known as the support of $f$. If $h$ is a homomorphism from $C(X)$ or $C^*(X)$ into $C(Y)$, then the image of a bounded function on $X$ is a bounded function on $Y$ under $h$. A space is called pseudocompact if $C(X)=C^*(X)$. A maximal ideal $M$ in $C(X)$ is called real maximal if $C(X)/ M$ is isomorphic with $\mathbb{R}$, otherwise is called hyper-real. A maximal ideal is called fixed if there exists $x\in X$ such that $f(x) =0$, for all $f\in M$, usually denoted as $M_x$, Every fixed maximal ideal is real. However converse may not be true. A space is realcompact if every real maximal ideal in $C(X)$ is fixed. B. Mitra and S.K. Acharyya introduced the ring $\chi(X)$ in their paper \cite{ma05}. $\chi(X)$ is the smallest subring of $C(X)$ consisting of $C^*(X)$ and $C_{H}(X)$, where $C_{H}(X)=\{f\in C(X)|cl_{X}(X\backslash Z(f))$ is hard in $X\}$. A subset $H$ of $X$ is hard in $X$ if it is closed in $X\cup cl_{\beta X}$.

 They proved the following theorem

 \begin{theorem}
 A space $X$ is nearly pseudocompact if and only if $\chi(X)=C^*(X)$.
\end{theorem}

In the same paper \cite{ma05}, Mitra and Acharyya defined hard pseudocompact space. A space is said to be hard pseudocompact if $C(X)=\chi(X)$. Then it is evident, hard pseudocompactness and nearly pseudocompactness together imply pseudocompactness of a space $X$ and vice versa.. Later Ghosh and Mitra in \cite{gm12} worked in detail over hard pseudocompact spaces and their properties.

Henriksen and Mitra intoduced Strongly real maximal ideal (in brief SRM ideal) in \cite{hm05}. A maximal ideal $M$ of $C(X)$ is called SRM ideal if there exists $g\notin M$ such that $fg\in C^*(X)$ for all $f \in C(X)$. There is no in general any connection between real maximal ideal and SRM ideal. Infact not all fixed maximal ideal is SRM. If $X$ is locally compact, all fixed maximal ideals are SRM.  However SRM ideals helps to characterize nearly realcompact space.

\begin{theorem}
A space $X$ is nearly realcompact if and only if every SRM ideal is fixed.
\end{theorem}

In $C(X)$, more generally in commutative ring with 1, if $\mathcal{M}(X)$ is the collection of all maximal ideal of $C(X)$, naturally topologized with Hull-kernel topology. then $\mathcal{M}(X)$ with this topology is called structure space of $C(X)$. In this structure space $\mathcal{X}_f= \{ M\in \mathcal{M}(X): f\in M\}$ forms a base for closed sets and the structure space turns out to be compact $T_1$ space. In general structure space of a commutative ring with 1 may not be Hausdorff but the structure space of $C(X)$ or any subring of $C(X)$ containing $C^*(X)$, henceforth to be referred as intermediate subrings of $C(X)$,  turns out to be Hausdorff.

Byun and Watson in \cite{bw91} discussed different properties an intermediate subrings of $C(X)$ almost similar to that of $C(X)$. As for instance, any intermediate subring $A(X)$ of $C(X)$ is a lattice ordered ring. Any maximal ideal of $A(X)$ is absolutely convex. $A(X)/M$ is a totally ordered field containing $\mathbb{R}$ as a totally ordered subfield. For each $x \in X$, $M_x^A=\{f\in A(X): f(x)=0\}$ is precisely the collection of fixed maximal ideal of $A(X)$. A maximal ideal of $A(X)$ is real if $A(X)/M$ is isomorphic with $\mathbb{R}$. Every fixed maximal ideal of $A(X)$ is real. Redlin and Watson in \cite{rw87} defined a space to be $A$-compact if every real maximal ideal of $A(X)$ is fixed and proved the following theorem [\cite{rw87}, Theorem 7].

\begin{theorem}\label{delhi}
Let $X$ be $A$-compact and $Y$ be $B$-compact. If $A(X)$ is isomorphic with $B(X)$, then $X$ is homeomorphic with $Y$.
\end{theorem}

\section{Main results}

Let $A(X)$ be an intermediate subring of $C(X)$. Let $\mathcal{M}(A)$ be a family of maximal ideals in $A(X)$. Let $\mathcal{P}$ be a subset of $\mathcal{M}(A)$ . A maximal ideal $M$ is called $\mathcal{P}_A$-maximal ideal if $M\in \mathcal{P}$, otherwise $M$ is called non-$\mathcal{P}_A$-maximal ideal. We call a space $X$ to be $\mathcal{P}_A$-compact if every $\mathcal{P}_A$-maximal ideal is fixed. For $A(X) = C(X)$, we shall simply write $\mathcal{M}(A)$ by $\mathcal{M}(X)$, $\mathcal{P}_C$-maximal ideal by $\mathcal{P}$-maximal and $\mathcal{P}_C$-compact by $\mathcal{P}$-compact.  It is clear that every compact space is $\mathcal{P}$-compact.

\begin{remark}
Notion of $\mathcal{P}_A$-maximal ideal may not be algebraic objects that is may not be preserved under isomorphism. Likewise $\mathcal{P}_A$-compactnes may not be a topological property, that is, may not be preserved under homeomorphism.  For example if we take $\mathcal{P}$ to be the collection of fixed maximal ideals, then $\mathcal{P}$-maximal ideals are not algebraic objects. Here every space turns out to be $\mathcal{P}$-compact space and $\mathcal{P}$-compactness is not a topological property.
\end{remark}

Now let for any space $X$ and $Y$, $A(X)$ and $B(Y)$ be intermediate subrings of $C(X)$ and $C(Y)$ respectively. Let  $\sigma:A(X) \to B(Y) $ be an isomorphism. Then we can lift $\sigma$ to $\sigma^*: \mathcal{P}(\mathcal{M}(A)) \to \mathcal{P}(\mathcal{M}(B)) $ defined by $\sigma^*(\mathcal{E}) = \{\sigma(M): M\in \mathcal{E}\}$.  $\mathcal{Q} \subseteq \mathcal{M}(B)$ is said to be $\sigma$-conjugate of $\mathcal{P} \subseteq \mathcal{M}(A)$ if $\sigma^*(\mathcal{P})=\mathcal{Q}$.   $\mathcal{Q}$ is said to be conjugate of $\mathcal{P}$ if $\mathcal{Q}$ is $\sigma$-conjugate of $\mathcal{P}$  for any isomorphism $\sigma$ between $A(X)$ and $B(Y)$. It is clear that if $\mathcal{Q}$ is conjugate of $\mathcal{P}$ if and only if $\mathcal{P}$ is conjugate of $\mathcal{Q}$.

A subset $\mathcal{P}$ of $\mathcal{M}(A)$ is called algebraic if whenever $A(X)$ is isomorphic with $B(Y)$,for any space $Y$, there exists a subset $\mathcal{Q}$ of $\mathcal{M}(B)$ which is conjugate of $\mathcal{P}$ If $\mathcal{P}$ is algebraic subset of $\mathcal{M}(X)$, then we will keep the same symbol $\mathcal{P}$ for its conjugate, as was justified by theorem \ref{sodepur} and theorem \ref{bangalore} below.

We hereby instantiate few examples of algebraic sets of maximal ideals in $C(X)$.

\begin{example}
\begin{enumerate}

\item If $\mathcal{P}$ is the collection of all maximal ideals of $C(X)$. Then $\mathcal{P}$ is trivially an algebraic set. Then every maximal ideal is then $\mathcal{P}$-maximal ideal and hence $\mathcal{P}$-compact spaces are precisely the compact spaces.

\item If $\mathcal{P}$ is the collection of all maximal ideal of $C(X)$ such that $C(X)/M$ is isomorphic with the field of reals. Then again $\mathcal{P}$ is an algebraic set. Then we know that a maximal ideal is $\mathcal{P}$-maximal ideal if and only if it real maximal ideal introduced by Hewitt in [\cite{h48}] and $\mathcal{P}$-compact spaces are precisely realcompact spaces [\cite{h48}].

\item If $\mathcal{P}$ is the collection of all maximal ideal satisfying the following condition: there exists $f$ outside $M$ such that $fg\in C^*(X)$ for all $g \in C(X)$. Then $\mathcal{P}$ is also algebraic as $C^*(X)$ is invariant under any isomorphism from $C(X)$ to $C(Y)$ .Then $\mathcal{P}$-maximal ideals are precisely the SRM-ideals introduced by Henriksen and Mitra in [\cite{hm05}, Theorem 2.12] and $\mathcal{P}$-compact spaces are precisely the nearly realcompact spaces [\cite{hm05}, Theorem 2.9]
\end{enumerate}
\end{example}

We now topologize $\mathcal{M}(A)$ with the hull-Kernel topology with $\{\mathcal{A}_{f}|f\in A(X)\}$ as its base for closed sets in $\mathcal{M}(A)$, where $\mathcal{A}_{f}=\{M\in\mathcal{M}(A)| \ f\in M\}$  and consider the subspace $\mathcal{P}$ of $\mathcal{M}(A)$. The base for closed sets in $\mathcal{P}$ is of the form $\mathcal{A}_{f}\cap \mathcal{P}$, for all $f \in A(X)$.

\begin{lemma}
If $\psi$ be an isomorphism between $A(X)$ and $B(Y)$ and $\mathcal{P}\subseteq \mathcal{M}(A)$ then $\mathcal{P}$ is homeomorphic with $\psi^*{\mathcal{P}}$.
\end{lemma}

\begin{proof}
Let $\psi:A(X)\rightarrow B(Y)$ be the isomorphism. Define $\phi:\mathcal{P}\rightarrow \psi^*{\mathcal{P}}$
by $\phi(M)=\psi(M).$ Therefore $\phi$ is clearly well-defined.
$\phi$ is evidently one-one and onto. Now we have to show that $\phi$ is closed
and continuous. Here $\mathcal{B}_{f}\cap\mathcal{\psi^*\mathcal{P}}$ is a
basic closed set in $\psi^*{\mathcal{P}}$ and $\phi^{-1}(\mathcal{B}_{f} \cap \psi^*\mathcal{P})=\phi^{-1}(\mathcal{B}_{f})\cap\phi^{-1}(\psi^*{\mathcal{P}})=\mathcal{A}_{\psi^{-1}(f)}\cap\mathcal{P}$,
which is a basic closed set in $\mathcal{P}$. Again $\phi(\mathcal{A}_{f}\cap\mathcal{P})=\mathcal{B}_{\psi(f)}\cap\mathcal{P}$.
Hence $\phi$ is closed and continuous. Therefore $\mathcal{P}$
and $\psi^*\mathcal{P}$ are homeomorphic.
\end{proof}

For $\mathcal{P}\subseteq \mathcal{M}(A)$, let $N_\mathcal{P}^A (X)$ be the set of all those points $x$ in $X$ for which $M_x^A$ is not $\mathcal{P}_A$-maximal ideal. We call $X$ to be locally-$\mathcal{P}_A$ if $N_\mathcal{P}^A (X) = \emptyset$. For $A(X) = C(X)$, we denote  $N_\mathcal{P}^A (X) = N_\mathcal{P} (X)$ and locally-$\mathcal{P}_A$ by locally-$\mathcal{P}$. For example if $\mathcal{P}$ be the class of all real maximal ideals of $C(X)$ , then $X$ is locally $\mathcal{P}$. If we take $\mathcal{P}$ to be family of all SRM ideal in $C(X)$, then $X$ is not locally $\mathcal{P}$.

\begin{lemma}
If $X$ is $\mathcal{P}_A$-compact then $X\backslash\mathcal{N}_{P}^A(X)$ and
$\mathcal{P}$ are homeomorphic.
\end{lemma}

\begin{proof}
Let $X$ be $\mathcal{P}_A$-compact, for $\mathcal{P}\subseteq \mathcal{M}(A)$. Then $\mathcal{P}=\{M_{x}^A|x\in X\backslash\mathcal{N}_{\mathcal{P}}^A(X)\}$.
Now define $\sigma:X\backslash\mathcal{N}_{\mathcal{P}}^A(X)\rightarrow\mathcal{P}$ by $\sigma(x)=M_{x}^A$. Then $\sigma$ is bijective. Let $y\in\sigma^{-1}(\mathcal{A}_{f}\cap\mathcal{P})\Leftrightarrow\sigma(y)\in\mathcal{A}_{f}\cap\mathcal{P}\Leftrightarrow M_{y}^A\in\mathcal{A}_{f}\cap\mathcal{P}\Leftrightarrow M_{y}^A\in\mathcal{A}_{f}$
and $M_{y}^A\in\mathcal{P}\Leftrightarrow f\in M_{y}^A$ and $y\notin\mathcal{N}_{\mathcal{P}}^A(X)\Leftrightarrow f(y)=0$
and $y\notin\mathcal{N}_{\mathcal{P}}^A(X)\Leftrightarrow y\in Z(f)$ and $y\notin\mathcal{N}_{\mathcal{P}}^A(X)\Leftrightarrow y\in Z(f)\cap(X\backslash\mathcal{N}_{\mathcal{P}}^A (X))$.
Therefore $\sigma^{-1}(\mathcal{A}_{f}\cap\mathcal{P})=Z(f)\cap(X\backslash\mathcal{N}_{\mathcal{P}}^A(X))$.
Also $\sigma(Z(f)\cap(X\backslash\mathcal{N}_{\mathcal{P}}^A(X)))=\mathcal{A}_{f}\cap\mathcal{P}$.
Hence $\sigma$ is a homeomorphism i.e $X\backslash\mathcal{N}_{\mathcal{P}}^A(X)$
and $\mathcal{P}$ are homeomorphic.
\end{proof}

Hence the following theorem is immediate.

\begin{theorem}\label{dgp}
Let $\psi: A(X) \to B(Y)$ be an isomorphism and $\mathcal{P} \subseteq \mathcal{M}(A)$. If $X$ is $\mathcal{P}_A$-compact and $Y$ is $\mathcal{\psi^*P}_B$-compact, then $X\backslash\mathcal{N}_{\mathcal{P}}^A(X)$ is homeomorphic with $Y\backslash\mathcal{N}_{\mathcal{P}}^B(Y)$.
\end{theorem}

\begin{proof}
As $X\backslash\mathcal{N}_{P}^A(X)$ is homeomorphic with $\mathcal{P}$ and $\mathcal{P}$ is homeomorphic with $\psi^*(P)$. Further as, $Y$ is $\mathcal{\psi^*P}_B$-compact, $\psi^*(P)$ is homeomorphic with $Y\backslash\mathcal{N}_{P}(Y)$, by transitivity, $X\backslash\mathcal{N}_{P}^A(X)$ and $Y\backslash\mathcal{N}_{P}^B(Y)$
\end{proof}

As already mentioned before that $\mathcal{P}$-compact or $\psi^*\mathcal{P}$-compact may not be topological property, But if we take $\mathcal{P}$ to be algebraic subset of $\mathcal{M}(X)$, then following theorem shows that $\mathcal{P}$-compactness is a topological property.

\begin{theorem}\label{sodepur}
Let $\mathcal{P}$ be a subset of $\mathcal{M}(X)$ which is algebraic. Then $\mathcal{P}$-compact is a topological property.
\end{theorem}

\begin{proof}
Let $\mathcal{P}$ be an algebraic set in $\mathcal{M}(X)$. Suppose $X$ is homeomorphic with $Y$.  Let $\sigma$ be the homeomorphism. $\psi_\sigma: C(X) \to C(Y)$ defined by $\psi_\sigma (f) = f\circ \sigma^{-1}$ is an isomorpism. Let $M$ be a $\mathcal{P}$-maximal ideal in $C(Y)$, then there exists a unique $\mathcal{P}$-maximal ideal $N\in C(X)$ such that $\psi^*(N) = M$ as by definition of algebraic set, $\psi^*(\mathcal{P}) = \mathcal{P}$. Since $X$ is $\mathcal{P}$-compact, $N=N_x$. Let $f\in M_{\sigma (x)}$. Then $f(\sigma (x))=0$. Thus $f\circ \sigma \in N_x$,  $\psi_{\sigma}(f\circ \sigma)\in M$. But $\psi_{\sigma}(f\circ \sigma)= f\circ \sigma \circ \sigma^{-1} = f$. So $f \in M$. Thus $M_{\sigma (x)} \subseteq M$ and due to maximality
$M = M_{\sigma(x)}$. Hence every $\mathcal{P}$-maximal ideal is fixed. So $Y$ is $\mathcal{P}$-compact.

\end{proof}

\begin{theorem}\label{bangalore}
Let $\mathcal{P}$ be an algebraic set. $X$ is locally $\mathcal{P}$ and $X$ is homeomorphic with $Y$, then $Y$ is locally $\mathcal{P}$.
\end{theorem}

\begin{proof} Let $M_y^Y$ be a fixed maimal ideal in $C(Y)$. As $X$ is homeomorphic with $Y$, by the proof of the above theorem,  this homomorphism induces a canonical isomorphism between $C(X)$ and $C(Y)$ which takes back the fixed maximal ideal $M_y^Y$ to a fixed maximal ideal of $C(X)$. As $X$ is locally $\mathcal{P}$ and $\mathcal{P}$ is algebraic, $M_y^Y\in \mathcal{P}$. Hence $Y$ is locally $\mathcal{P}$.
\end{proof}

The above two theorems hereby justify our convention to denote conjugates of algebraic sets by same symbol.

\

The next theorem is important to us and is direct consequence of theorem \ref{dgp}, as special case for $A(X)=C(X)$.

\begin{theorem}
If $X$, $Y$ are $\mathcal{P}$-compact spaces for an algebraic set $\mathcal{P}$ and $C(X)$ is isomorphic with
$C(Y)$ then \textup{$X\backslash\mathcal{N}_{\mathcal{P}}(X)$ and $Y\backslash\mathcal{N}_{\mathcal{P}}(Y)$.
are homeomorphic.}
\end{theorem}

The following theorem estblishes that if $\mathcal{P}$ is an algebraic set containg all fixed maximal ideal, the if $X$ is $\mathcal{P}$-compact, $C(X)$ determines $X$ too.

\begin{theorem}
Let $\mathcal{P}$ be an algebraic set and $C(X)$ is isomorphic with $C(Y)$. If $X$ and $Y$ both are locally $\mathcal{P}$ and $\mathcal{P}$-compact, then $X$ is homeomorphic with $Y$.
\end{theorem}

Proof directly follows from the previous theorem.

Now we can easily interpret all the three results as mentioned in the begining of this paper as a particular case.

\textbf{Banach-Stone theorem}: As we have already observed that if we choose $\mathcal{P}$ to be the collection of all maximal ideals, then $\mathcal{P}$ is an algebraic set. Clearly $X$ is locally-$\mathcal{P}$ and $\mathcal{P}$-compact spaces are precisely the compact spaces. So if $X$ is compact space, then $C(X)$ determines $X$.

\textbf{E. Hewitt}: If we take $\mathcal{P}$ to be the collection of all real maximal ideals. Then $\mathcal{P}$ is an algebraic set and $X$ is also locally -$\mathcal{P}$. Then it follows that if $X$ is realcompact, then $C(X)$ determines $X$ too.

\textbf{M. Henriksen and B. Mitra}: If we choose $\mathcal{P}$ to be the collection of SRM ideals, then $\mathcal{P}$ is an algebraic set. Then $X$ is locally-$\mathcal{P}$ if and only $X$ is locally compact. Thus we conclude that if $X$ locally compact and nearly realcompact, then $C(X)$ also determines $X$.

 We call in general a maximal ideal $M$ of a commutative ring $A$ with unity is $B$-real maximal ideal, where $B$ is a subring of $A$ containing the unity of $A$ if $M \cap B$ is a maximal ideal of $B$. If $A(X)$ and $B(X)$ are intermediate subrings of $C(X)$ with $A(X)$ being subring of $B(X)$, then we call $X$ to be $B-A$-compact if every $A$-real maximal ideal of $B$ is fixed. It  is clear from [\cite{gj60}, theorem 7.9(c)], that real maximal ideals are precisely $C^*$-real maximal ideal of $C(X)$ and $C-C^*$-compact is precisely realcompact space. For any intermediate subring $A(X)$ of $C(X)$, real maximal ideal of $A(X)$ is precisely $C^*(X)$-real maximal ideal of $A(X)$ by [\cite{bw91}, corrollary 3.8]. So $C^*-A$-compact is precisely the $A$-compact spaces.

\textbf{Redlin and Watson}: If we take $\mathcal{P}$ to be the family of real maximal ideal of $A(X)$. Then $\mathcal{P}$ is an algebraic set. Clearly $X$ is locally $\mathcal{P}$. If $A(X)$ is isomorphic with $B(Y)$, $\mathcal{P}$-compactness of $X$ and $Y$ are respectively $A$-compactness of $X$ and $B$-compactness of $Y$. So theorem \ref{delhi} follows.

Now we shall try to build up another structurally similar example which is different from previous examples. In \cite{ma05} Mitra and Acharyya introduced a subring of $C(X)$ containing $C^*(X)$, referred as $\chi (X)$ which is the smallest subring of $C(X)$ containing $C^*(X)$ and $C_H (X)$. It is clear that if we choose $\mathcal{P}$ to be the collection of $\chi$-real maximal ideals of $C(X)$, then $\mathcal{P}$ is an algebraic set. $X$ is trivially locally $\mathcal{P}$. Then we have the following result.

\begin{theorem}
If $X$ and $Y$ are $\chi$-realcompact space, Then $C(X)$ isomorphic with $C(Y)$ implies that $X$ is homeomorphic with $Y$.
\end{theorem}

Since in hard pseudocompact space, $\chi(X) = C(X)$, $\chi$-real maximal ideals are precisely all maximal ideal and therfore it is obvious that $\chi$-realcompact and hard pseudocompact implies compactness and on the otherhand, if $X$ is nearly pseudocompact, $\chi (X) = C^*(X)$, $\chi$-real maximal ideals are precisely real maximal ideals and hence $\chi$-realcompact and nearly pseudocmpact implies realcompactness. \\

\
\textbf{Acknowledgements:} The authors sincerely acknowledge the support received
from DST FIST programme (File No. SR/FST/MSII/2017/10(C))

\end{document}